\documentclass{article}

\usepackage{amssymb}
\usepackage{amsmath}
\usepackage{amsthm}
\usepackage{ifthen}
\usepackage{a4wide}
\usepackage[english]{babel}

\newcommand{\cenlin}{\centerline}

\newcommand{\ovline}{\overline}

\newcommand{\vect}[1]{\mathbf{#1}}

\newcommand{\rarr}{\rightarrow}
\newcommand{\xrarr}[2]{\xrightarrow[#1]{#2}}

\newcommand{\Rarr}{\Rightarrow}
\newcommand{\LRarr}{\Leftrightarrow}

\newcommand{\backsl}{\backslash}

\newcommand{\nullset}{\varnothing}

\newcommand{\lesq}{\leqslant}
\newcommand{\greq}{\geqslant}

\newcommand{\veps}{\varepsilon}

\newcommand{\forceparindent}{\hskip 1.5em}

\DeclareMathOperator{\card}{card}
\DeclareMathOperator{\diam}{diam}

\newcounter{Formulanum}
\newcounter{Assumptionum}
\newcommand{\assump}[1]{\refstepcounter{Assumptionum}%
\label{#1}%
$\clubsuit$\arabic{Assumptionum}}
\newcommand{\assumpref}[1]{$\clubsuit$\ref{#1}}
\newcommand{\resetAssumpCount}[1]{\setcounter{Assumptionum}{#1}}

\newtheorem{corollary}{Corollary}
\newtheorem{definition}{Definition}

\newtheorem{proposition}{Proposition}

\newenvironment{exampledescript}{$\vartriangleleft$}{$\vartriangleright$}

\newcommand{\biauth}[1]{\textsc{#1}} 
\newcommand{\bitita}[1]{\textit{#1}} 
\newcommand{\bititb}[1]{\textrm{#1}} 
\newcommand{\bititu}[1]{\texttt{#1}} 
\newcommand{\bijour}[1]{\textrm{#1}} 
\newcommand{\bicoor}[3]{%
\ifthenelse{\equal{#1}{}}{}{\textbf{#1}}%
\ifthenelse{\equal{#2}{}}{}{:#2}%
\ifthenelse{\equal{#3}{}}{}{, #3}} 
\newcommand{\bipubl}[1]{\textrm{#1}} 
\newcommand{\bidoid}[1]{\texttt{doi:#1}} 

\begin{document}

\title{On sound ranging in some non-proper metric spaces}

\author{Sergij V. Goncharov\thanks{Faculty of Mechanics and Mathematics, Oles Honchar Dnipro National University,
72 Gagarin Avenue, 49010 Dnipro, Ukraine.
\textit{E-mail: goncharov@mmf.dnulive.dp.ua}}}

\date{October 2019}

\maketitle

\begin{abstract}
We consider the sound ranging, or source localization, problem ---
find the unknown source-point from known moments when the spherical wave
of linearly, with time, increasing radius reaches known sensor-points ---
in some non-proper metric spaces (closed ball is not always compact).
Under certain conditions we approximate the solution to arbitrary precision
by the iterative processes with and without a stopping criterion.
We also consider this problem in normed spaces with a strictly convex norm
when the sensors are dense on the unit sphere.

Appended is the implementation of the approximation algorithm in Julia language.
\end{abstract}

\cenlin{\small \textit{MSC2010:} Pri 41A65, Sec 54E35, 40A05, 46B20, 68W25}

\cenlin{\small \textit{Keywords:} sound ranging, localization, TDOA, approximation, algorithm, metric space, normed space}

\section*{Introduction}

\forceparindent
Let $(X;\rho)$ be a metric space, i.e. the set $X$ with the metric $\rho \colon X\times X \rarr \mathbb{R}_+$.
At unknown moment $t_0 \in \mathbb{R}$ the unknown ``source'' $\vect{s} \in X$ ``emits the (sound) wave'', which is the sphere
$\bigl\{ \vect{x} \in X \mid \rho(\vect{x}; \vect{s}) = v(t - t_0) \bigr\}$
at any moment $t \greq t_0$. We assume that ``sound velocity'' $v = 1$.

Let $\{ \vect{r}_i \}_{i \in I} \subseteq X$ be an indexed set of known ``sensors''.
For each sensor we also know the moment $t_i = t_0 + \rho (\vect{r}_i; \vect{s})$ when it was reached by the expanding wave.

The \textit{sound ranging} problem (SRP), also called \textit{source localization},
is to find $\vect{s}$ from $(\{ \vect{r}_i \}; \{ t_i \})$. Another name is \textit{time-difference-of-arrival} (TDOA) problem,
because to obtain $\vect{s}$, when possible,
it suffices to know the delays $t_i - t_j$ rather than ``absolute'' $t_i$.

SRPs, usually in Euclidean space and with noisy measurements, appear in many circumstances;
see e.g. \cite[1]{compagnoni2017}, \cite[9.1]{huang2004}, \cite[6]{rezazekavat2019}, \cite[1]{wu2019}
for further references on quite large and diverse literature.
For the first time they were considered and studied at least one century ago (see \cite[pp. 33--39]{aubin2014}).

{\ }

This paper is a generalization of \cite{goncharov2018b}, where we looked into SRP in \textit{proper}
(also called \textit{finitely compact} or \textit{Heine-Borel}) metric spaces,
in which any closed ball is compact. Without Euclidicity in general case, the classical approach
``solve equations $(t_i - t_0)^2 = \sum\limits_j (r^{(i)}_j - s_j)^2$,
where $t_0$ and coordinates $\{ s_j \}$ of $\vect{s}$ are unknowns''
(which we applied in \cite{goncharov2018a} investigating the noiseless SRP in $l_2$)
does not work even if the space is provided with some coordinates, because
they are not so easily ``extractable'' from the equations $t_i - t_0 = \rho( \vect{r}_i ; \vect{s} )$.
Instead, we described the ``approximating'' approach --- the iterative process that converges to the source.

Here we propose the variations of more or less the same approach, adjusted to regard non-properness of underlying space.
Being more general, they work in proper spaces as well.
For the sake of convenience, some content from \cite{goncharov2018b} is repeated with necessary modifications.

{\ }

We introduce some notions --- functions, sets, constructions etc. and the constraints they must satisfy ---
to formulate the approximation algorithm using them. Thus, if they are instantiated in any given space,
the algorithm can be implemented in that space accordingly; see Appendix.

Indeed, we are interested in an algorithm that does not require its ``executor'' (computer)
to be too far beyond mental and physical reach ``of sentient life in this universe'', particularly in terms of 
the elementary actions the executor can perform.
Our executor, for example, cannot run $\card X \greq 2^{\aleph_0}$ calculations in parallel
(then we would simply ``verify each $\vect{x} \in X$, if it is a solution, simultaneously'').
However, some data it needs to operate on may be considered as obtained from an ``oracle''
that is able e.g. to calculate the exact sum of an infinite series in a finite time.

{\ }

We deal with SRP in an ``empty'' space without other waves, reverberation, varying propagation velocity,
imprecise measurements etc. --- without ``physics''; this simplification makes the delays $t_i - t_j$ known exactly.
For certain non-negative function that depends on these delays,
we perform a \textit{root finding} of ``exclude \& enclose'' type (see \cite{byrne2008}), ---
we search for its unique zero instead of search for its extremum,
the latter would be an \textit{optimization} approach, of a kind widely used in solving SRP with noised data in Euclidean space,
particularly based on the maximum likelihood estimation, though other methods exist
(see e.g. \cite{alamedapineda2014}, \cite{bestagini2013}, \cite{chen2002}, \cite{diezgonzalez2019},
\cite{gillette2008}, \cite{huang2004}, \cite{pollefeys2008}, \cite{wu2019}).

{\ }

``$\bullet$'' Well-known statement (see \cite{giles1987}, \cite{jameson1974}, \cite{kolmogorov1975},
\cite{montesinos2015}, \cite{searcoid2007}), included for the sake of completeness.

``$\clubsuit$'' Additional assumption or constraint.

{\ }

One may feel that this paper (except for Appendix) should belong to 1920--30s.

\section{Preparations}

\forceparindent
We recall some basic terms and denotations to avoid ambiguity.

$\bullet$ ``2nd $\bigtriangleup$-inequality'':
$\forall \vect{x}, \vect{y}, \vect{z} \in X$
$\bigl| \rho(\vect{x}; \vect{z}) - \rho(\vect{z}; \vect{y}) \bigr| \lesq \rho(\vect{x}; \vect{y})$.

$\bullet$ Let $f\colon A \times B \rarr \mathbb{R}_+$. Then $\forall u, v \in B$:
$\bigl| \sup\limits_{w \in A} f(w; u) - \sup\limits_{w \in A} f(w; v) \bigr| \lesq
\sup\limits_{w \in A} \bigl| f(w; u) - f(w; v) \bigr|$.

As usual, $\vect{x}_k \xrarr{k \rarr \infty}{} \vect{y}$ means $\rho(\vect{x}_k; \vect{y}) \xrarr{k \rarr \infty}{} 0$.

$\bullet$ Continuity of metric: $\vect{x}_k \xrarr{k \rarr \infty}{} \vect{y}$ $\Rarr$
$\rho(\vect{x}_k; \vect{z}) \xrarr{k \rarr \infty}{} \rho(\vect{y}; \vect{z})$.

$B(\vect{c}; r) = \{ \vect{x} \in X \mid \rho(\vect{x}; \vect{c}) < r \}$,
$B[\vect{c}; r] = \{ \vect{x} \in X \mid \rho(\vect{x}; \vect{c}) \lesq r \}$,
and $S[\vect{c}; r]$ = $\{ \vect{x} \in X \mid \rho(\vect{x}; \vect{c}) = r \}$
denote \textit{open ball}, \textit{closed ball}, and \textit{sphere} with center $\vect{c}$ and of radius $r$.

$\bullet$ For any $B[\vect{c}; r]$ and any $\vect{a} \in X$ with $\rho(\vect{a}; \vect{c}) = d$ we have
$\forall \vect{x} \in B[\vect{c}; r]$: $d - r \lesq \rho(\vect{x}; \vect{a}) \lesq d + r$.

The \textit{closure} of the set $A \subseteq X$ is $\ovline{A} = \bigl\{ \vect{y} \in X \mid \exists \{ \vect{x}_k \}_{k \in \mathbb{N}} \subseteq A \colon \vect{x}_k \xrarr{k \rarr \infty}{} \vect{y} \bigr\}$.

$A \subseteq X$ is said to be \textit{compact} if $\forall \{ \vect{x}_k \}_{k\in \mathbb{N}} \subseteq A$:
$\exists \{ \vect{x}_{k_l} \}_{l \in \mathbb{N}}$, $\exists \vect{x} \in A$: $\vect{x}_{k_l} \xrarr{l \rarr \infty}{} \vect{x}$.

$A \subseteq X$ is said to be \textit{relatively compact} if $\forall \{ \vect{x}_k \}_{k\in \mathbb{N}} \subseteq A$:
$\exists \{ \vect{x}_{k_l} \}_{l \in \mathbb{N}}$, $\exists \vect{x} \in X$: $\vect{x}_{k_l} \xrarr{l \rarr \infty}{} \vect{x}$.

$\bullet$ $A \subseteq X$ is relatively compact \textit{\textbf{iff}} $\ovline{A}$ is compact.

$\bullet$ If $A$ is compact, then any closed subset of $A$ is compact too.

The family of sets $\{ C_j \}_{j \in J}$, $C_j \subseteq X$, is said to be a \textit{cover} of $A \subseteq X$
if $A \subseteq \bigcup\limits_{j \in J} C_j$.

$\bullet$ The closed $A \subseteq X$ is compact \textit{\textbf{iff}} any open cover of $A$ has finite subcover.

$A \subseteq X$ is called \textit{bounded} if
$\mathrm{diam}\, A := \sup\limits_{\vect{x}, \vect{y} \in A} \rho(\vect{x}; \vect{y}) < \infty$.

$\bullet$ $A$ is bounded $\LRarr$ $\exists B[\vect{c}; r] \supseteq A$.

The norm $\| \cdot \|$ of a normed space $(X; \| \cdot \|)$ is called \textit{strictly convex}
if spheres do not contain segments: $\forall \vect{x}, \vect{y} \in S[\theta; 1]$
(where $\theta$ is zero of $X$ as linear vector space) such that $\vect{x} \ne \vect{y}$,
and $\forall \lambda \in (0;1)$: $\| \lambda \vect{x} + (1 - \lambda) \vect{y} \| < 1$.

$\bullet$ The norm is strictly convex \textit{\textbf{iff}} $\bigtriangleup$-inequality becomes equality
only for positively proportional elements: $\forall \vect{x}, \vect{y} \in X$,
if $\| \vect{x} + \vect{y} \| = \| \vect{x} \| + \| \vect{y} \|$
and $\vect{x} \ne \theta$, then $\vect{y} = \lambda \vect{x}$ for some $\lambda \greq 0$.

{\ }

Now we proceed to the SRP.
The source $\vect{s} \in X$ and the emission moment $t_0 \in \mathbb{R}$ are unknown.

{\ }

\assump{asmpSensBound}. The set of sensors $R = \{ \vect{r}_i \}_{i \in I}$ is bounded.

{\ }

These sensors and the moments

\cenlin{$t_i = t_0 + \rho(\vect{r}_i; \vect{s})$, $i \in I$}

\noindent
define the SRP
$\bigl( \{ \vect{r}_i \} ; \{ t_i \} \bigr)$.
Any pair $(\vect{s'}; t') \in X \times \mathbb{R}$ satisfying the set of equations

\cenlin{$t_i = t' + \rho( \vect{r}_i; \vect{s}' )$, $i \in I$}

\noindent
is a \textit{solution} of this SRP.
Or, we call $\vect{s}' \in X$ itself a solution when $t_i - \rho(\vect{r}_i; \vect{s}') \equiv const$ ($ = t'$).

{\ }

\assump{asmpSolUnique}. The solution $\vect{s}$ of the SRP $\bigl( \{ \vect{r}_i \} ; \{ t_i \} \bigr)$ is unique in $X$.

{\ }

This is not the general case; for example, when all $\vect{r}_i$ are equal, any $\vect{s}' \in X$ is a solution,
with $t' = t_1 - \rho( \vect{s}'; \vect{r}_1 )$.
On the other hand, in some spaces we can ensure such uniqueness by placing the sensors appropriately:
in $l_2$ we take $\vect{r}_2 = \theta$, $\vect{r}_i = \vect{e}_{i-2}$ for $i \greq 3$, and $\vect{r}_1 = -\vect{e}_1$
(\cite[Prop. 4]{goncharov2018a}).

\begin{definition}\label{defBackwardMoments}
For any $\vect{x} \in X$ the \emph{backward moments}
$\tau_i(\vect{x}) := t_i - \rho( \vect{x}; \vect{r}_i )$, $i \in I$.
\end{definition}

$\tau_i(\vect{x})$ is the moment when the wave must be emitted from $\vect{x}$ to reach $\vect{r}_i$ at the moment $t_i$.

Since \assumpref{asmpSensBound} implies $R \subseteq B[\vect{c}; r]$, we have $\forall i \in I$:
$|\tau_i(\vect{x})| \lesq |t_0| + \rho(\vect{s}; \vect{r}_i) + \rho(\vect{x}; \vect{r}_i) \lesq$

\cenlin{$\lesq |t_0| + \rho(\vect{s}; \vect{c}) + \rho(\vect{c}; \vect{r}_i) +
\rho(\vect{x}; \vect{c}) + \rho(\vect{c}; \vect{r}_i) \lesq |t_0| + 2r + \rho(\vect{s}; \vect{c}) + \rho(\vect{x}; \vect{c}) =
T(\vect{x})$}

\noindent
where $T(\vect{x}) \in \mathbb{R}_+$ does not depend on $i$.

\begin{definition}\label{defDefectSup}
For any $\vect{x} \in X$ the \emph{defect}
$D_{\infty}(\vect{x}) := \sup\limits_{i,j \in I} \bigl| \tau_i(\vect{x}) - \tau_j(\vect{x}) \bigr|$.
\end{definition}

\begin{definition}\label{defDefectSum}
If $R$ is finite ($I = \{ 1; ...; n\}$) or countable ($I = \mathbb{N}$), then for any $\vect{x} \in X$ the \emph{defect}

\cenlin{$D_1(\vect{x}) := \sum\limits_{i \in I} p_i \bigl| \tau_i(\vect{x}) - \sum\limits_{j \in I} p_j \tau_j(\vect{x}) \bigr|$}

\noindent
where $\{ p_i \}_{i \in I}$ is some fixed probability distribution on $I$ without zeros:
$p_i > 0$ and $\sum\limits_{i \in I} p_i = 1$.
\end{definition}

For the sake of definiteness, consider e.g. $p_i \equiv \frac{1}{n}$ when $I = \{ 1; ...; n \}$
and $p_i = 2^{-i}$ when $I = \mathbb{N}$.

$D_{\infty}(\vect{x}) = \sup\limits_{i \in I} \tau_i(\vect{x}) - \inf\limits_{i \in I} \tau_i(\vect{x})$, and
$D_1(\vect{x}) \lesq \sum\limits_{i \in I} p_i \sum\limits_{j \in I} p_j \bigl| \tau_i(\vect{x}) - \tau_j(\vect{x}) \bigr| \lesq
D_{\infty}(\vect{x}) \lesq 2 T(\vect{x})$.

Hereinafter the usage of $D_1$ requires $R$ to be finite or countable,
while there is no such restriction for $D_{\infty}$.
We consider $D_1$ because in some circumstances its calculation may be ``easier''.

By $D(\vect{x})$ we denote either $D_{\infty}(\vect{x})$ or $D_1(\vect{x})$,
though we assume the same choice for all $\vect{x} \in X$.
Particularly, Propositions \ref{propDefectZeroSol}--\ref{propNegBallSmallEnough} that follow could begin with
``Let $D = D_{\infty}$ or $D = D_1$''.

\begin{proposition}\label{propDefectZeroSol}
$\vect{s}' \in X$ is the solution of the SRP \textbf{iff} $D(\vect{s}') = 0$.
\end{proposition}

\begin{proof}
If $\vect{s}'$ is such solution, then $\tau_i(\vect{s}') \equiv t'$ $\Rarr$ $\tau_i(\vect{s}') - \tau_j(\vect{s}') \equiv 0$
$\Rarr$ $D(\vect{s}') = 0$.
Contrariwise, $D(\vect{s}') = 0$
(along with $p_i > 0$ for $D_1$) implies $\tau_i(\vect{s}') \equiv t'$, and $(\vect{s}'; t')$ is the solution.
\end{proof}

\begin{corollary}\label{corDefectSingleZero}
$D(\vect{x})$ has exactly one zero in $X$, at $\vect{x} = \vect{s}$.
\emph{(Follows from \assumpref{asmpSolUnique}.)}
\end{corollary}

\begin{proposition}\label{propDefectDiffUppEstim}
$\forall \vect{x}, \vect{y} \in X$:
$\bigl| D(\vect{x}) - D(\vect{y}) \bigr| \lesq 2 \rho(\vect{x}; \vect{y})$.
\end{proposition}

\begin{proof}
$\bigl| \tau_k(\vect{x}) - \tau_k(\vect{y}) \bigr| =
\bigl| t_k - \rho(\vect{x}; \vect{r}_k) - t_k + \rho(\vect{y}; \vect{r}_k) \bigr| \lesq \rho(\vect{x}; \vect{y})$, therefore
$\bigl| D_{\infty}(\vect{x}) - D_{\infty}(\vect{y}) \bigr| \lesq$

\cenlin{$\lesq \sup\limits_{i,j \in I} \bigl| \tau_i(\vect{x}) -
\tau_j(\vect{x}) - \tau_i(\vect{y}) + \tau_j(\vect{y}) \bigr|
\lesq \sup\limits_{i,j \in I} \bigl| \tau_i(\vect{x}) - \tau_i(\vect{y}) \bigr| +
\sup\limits_{i,j \in I} \bigl| \tau_j(\vect{x}) - \tau_j(\vect{y}) \bigr| =$}

\cenlin{$= \sup\limits_{i \in I} \bigl| \tau_i(\vect{x}) - \tau_i(\vect{y}) \bigr| +
\sup\limits_{j \in I} \bigl| \tau_j(\vect{x}) - \tau_j(\vect{y}) \bigr| \lesq 2 \rho(\vect{x}; \vect{y})$}

Similarly, $\bigl| D_1(\vect{x}) - D_1(\vect{y}) \bigr| =
\Bigl| \sum\limits_{i \in I} p_i \bigl[ \, \bigl| \tau_i(\vect{x}) - \sum\limits_{j \in I} p_j \tau_j (\vect{x}) \bigr| -
\bigl| \tau_i(\vect{y}) - \sum\limits_{j \in I} p_j \tau_j(\vect{y}) \bigr| \, \bigr] \Bigr| \lesq$

\cenlin{$\lesq \sum\limits_{i \in I} p_i \Bigl| \, \bigl| \tau_i(\vect{x}) - \sum\limits_j (...) \bigr| -
\bigl| \tau_i(\vect{y}) - \sum\limits_j(...) \bigr| \, \Bigr| \lesq
\sum\limits_{i \in I} p_i \Bigl| \, \bigl[ \tau_i(\vect{x}) - \tau_i(\vect{y}) \bigr] -
\sum\limits_{j \in I} p_j \bigl[ \tau_j(\vect{x}) - \tau_j(\vect{y}) \bigr] \, \Bigr| \lesq$}

\hfil
$\lesq \sum\limits_{i \in I} p_i \Bigl[ \, \bigl| \tau_i(\vect{x}) - \tau_i(\vect{y}) \bigr| +
\sum\limits_{j \in I} p_j \bigl| \tau_j(\vect{x}) - \tau_j(\vect{y}) \bigr| \, \Bigr] \lesq
\sum\limits_{i \in I} p_i \bigl[ \rho(\vect{x}; \vect{y}) + \sum\limits_{j \in I} p_j \rho(\vect{x}; \vect{y}) \bigr] =
2 \rho(\vect{x}; \vect{y})$\hfil \end{proof}

\noindent
--- $D(\cdot)$ is a Lipschitz function (\cite[9.4]{searcoid2007}).

\begin{corollary}\label{corDefectUnifCont}
$D(\vect{x})$ is uniformly continuous on $X$.
\end{corollary}

\begin{definition}\label{defSmallDefectSmallDistProp}
Let $A \subseteq X$. We say that $A$ has \emph{SDN property} (``if-Small-Defect-then-Near'') when

\cenlin{$\vect{s} \in A$ $\Rarr$ $\forall \delta > 0$:
$\inf\limits_{\vect{x} \in A \backsl B(\vect{s}; \delta)} D(\vect{x}) > 0$}
\end{definition}

\noindent
(Here $\inf \nullset = +\infty$.) I.e. $\vect{s} \in A$ implies $\forall \delta > 0$
$\exists \veps \bigl( = \inf\limits_{\vect{x} \in A \backsl B(\vect{s}; \delta)} D(\vect{x}) \bigr) > 0$:
if $\vect{x} \in A$ and $D(\vect{x}) < \veps$, then $\rho(\vect{x}; \vect{s}) < \delta$.
Clearly, any subset of such $A$ has SDN property as well.

\begin{proposition}\label{propPrecompactIsSDN}
Any relatively compact set in $X$ has SDN property.
\end{proposition}

\begin{proof}
Let $A$ be such set, $\vect{s} \in A$, then $\forall \delta > 0$ let $G = A \backsl B(\vect{s}; \delta)$,
$\veps = \inf\limits_{\vect{x} \in G} D(\vect{x})$; we claim that $\veps > 0$. Indeed,
$\forall k \in \mathbb{N}$ $\exists \vect{x}_k \in G$: $\veps \lesq D(\vect{x}_k) \lesq \veps + \frac{1}{k}$.
Due to relative compactness $\exists \{ \vect{x}_{k_l} \}_{l \in \mathbb{N}}$, $\exists \vect{x} \in X$:
$\vect{x}_{k_l} \xrarr{l \rarr \infty}{} \vect{x}$. $\lim\limits_{l \rarr \infty} D(\vect{x}_{k_l}) = \veps$,
and it follows from continuity of $D(\cdot)$ that $D(\vect{x}) = \veps$.
Now, continuity of metric implies $\rho(\vect{x}; \vect{s}) \greq \delta$, hence by Cor. \ref{corDefectSingleZero} $\veps > 0$.
\end{proof}

\begin{corollary}\label{corCompactIsSDN}
Any compact set in $X$ has SDN property.
\end{corollary}

\textbf{Non-SDN example.} Let $(X; \rho)$ be $l_2$. By $E = \{ \vect{e}_i \}_{i \in \mathbb{N}}$ we denote the usual
orthonormal basis of $l_2$, that is, the coordinates $e^{(i)}_j = \delta_{ij}$.
Let $I = \mathbb{Z} \backsl \{0\}$, $R = \{ \vect{r}_i \}_{i \in I} := E \cup (-E) =
\{ ...; -\vect{e}_2; -\vect{e}_1; \vect{e}_1; \vect{e}_2; ... \}$ --- $\forall i \in \mathbb{N}$ $\vect{r}_i = \vect{e}_i$,
$\vect{r}_{-i} = - \vect{e}_i$.

We claim that $A = B[\theta; 1]$ does not have SDN property.

\begin{exampledescript}
Obviously, \assumpref{asmpSensBound} is satisfied.
Let $\vect{s} = \theta$ and $t_0 = 0$, then $\forall \vect{r}_i \in R$ $t_i \equiv 1$;
\assumpref{asmpSolUnique} holds because, assuming that $\vect{x} \in l_2$ is a solution,
it must be equidistant from all $\vect{r}_i$,
and $\forall i \in \mathbb{N}$ $\rho(\vect{x}; \vect{r}_i) = \rho(\vect{x}; \vect{r}_{-i})$
$\LRarr$ $\sum\limits_{j \in \mathbb{N}, j \ne i} x_j^2 + (x_i - 1)^2 = \sum\limits_{j \in \mathbb{N}, j \ne i} x_j^2 + (x_i + 1)^2$
$\LRarr$ $|x_i - 1| = |x_i + 1|$ $\LRarr$ $x_i = 0$ i.e. $\vect{x} = \vect{s}$.

Consider $\{ \vect{x}_n \}_{n \in \mathbb{N}}$
with the coordinates $x^{(n)}_j = 1 / \sqrt{n}$ for $j \lesq n$ and $x^{(n)}_j = 0$ for $j > n$.

$\rho(\vect{x}_n; \theta) = \sqrt{\sum\limits_{j=\ovline{1,n}} \frac{1}{n}} = 1$,
so $\vect{x}_n \in A \backsl B(\theta; \delta)$ for any $\delta \in (0;1]$.
$\forall i \in \mathbb{N}$

\cenlin{$\rho(\vect{x}_n; \vect{r}_i) = \begin{cases}
\sqrt{ \bigl[ \sum\limits_{j = \ovline{1,n}, j\ne i} \frac{1}{n} \bigr] + (1 - \frac{1}{\sqrt{n}})^2} =
\sqrt{2 - \frac{2}{\sqrt{n}}}, & i \lesq n,\\
\sqrt{\bigl[ \sum\limits_{j = \ovline{1,n}} \frac{1}{n} \bigr] + 1} = \sqrt{2}, & i > n,
\end{cases}$
$\rho(\vect{x}_n; \vect{r}_{-i}) = \begin{cases}
\sqrt{2 + \frac{2}{\sqrt{n}}}, & i \lesq n,\\
\sqrt{2}, & i > n.
\end{cases}$}

Therefore $D_{\infty}(\vect{x}_n) = \sup\limits_{i \in I} \tau_i(\vect{x}_n) - \inf\limits_{i \in I} \tau_i(\vect{x}_n) =
\sup\limits_{i \in I} \bigl[ 1 - \rho(\vect{x}_n; \vect{r}_i) \bigr] -
\inf\limits_{i \in I} \bigl[ 1 - \rho(\vect{x}_n; \vect{r}_i) \bigr] =
\sqrt{2 + \frac{2}{\sqrt{n}}} - \sqrt{2 - \frac{2}{\sqrt{n}}} \xrarr{n \rarr \infty}{} 0$,
which implies $\forall \delta \in (0;1]$: $\inf\limits_{\vect{x} \in A \backsl B(\theta; \delta)} D_{\infty}(\vect{x}) = 0$.
\end{exampledescript}

Take $\vect{x}_n = \frac{\delta}{\sqrt{n}} \sum\limits_{i=1}^n \vect{e}_i$ to show that, with this $R$,
$\forall r > 0$ $B[\theta; r]$ in $l_2$ is non-SDN.
Cf. the reasonings from the end of Section \ref{secDenseOnSphereNormSpace}, though,
where $A$ is of the same kind, but $R = S[\theta; 1]$.

\begin{proposition}\label{propBallTestReverse}
If $\vect{s} \in B[\vect{c}; r]$, then $D(\vect{c}) \lesq 2r$.
\end{proposition}

\begin{proof}
If $\vect{s} \in B[\vect{c}; r]$,
then $\rho(\vect{r}_i; \vect{c}) - r \lesq \rho(\vect{r}_i; \vect{s}) \lesq \rho(\vect{r}_i; \vect{c}) + r$, $i \in I$ $\LRarr$

\cenlin{$\LRarr$ $t_0 \bigl( = t_i - \rho(\vect{r}_i; \vect{s}) \bigr) \in \bigl[ t_i - \rho(\vect{r}_i; \vect{c}) - r;
t_i - \rho(\vect{r}_i; \vect{c}) + r \bigr]$, $i \in I$}

\noindent
hence $t_0 \in \bigcap\limits_{i \in I} \bigl[ \tau_i(\vect{c}) - r; \tau_i(\vect{c}) + r \bigr] =: C$, so $C \ne \nullset$.
Recalling that $D_{\infty}(\vect{c}) = \sup\limits_{i \in I} \tau_i(\vect{c}) - \inf\limits_{i \in I} \tau_i(\vect{c})$,
it is easy to see that $C \ne \nullset$ \textbf{iff} $D_{\infty}(\vect{c}) \lesq 2r$.
And $D_1(\vect{c}) \lesq D_{\infty}(\vect{c})$.
\end{proof}

\begin{corollary}\label{corBallNegTest}
If $D(\vect{c}) > 2r$, then $\vect{s} \notin B[\vect{c}; r]$.
\end{corollary}

The balls $B[\vect{c}; r]$ that pass the test $D(\vect{c}) \lesq 2r$ are ``suspicious'':
more ``sophisticated'' tests may or may not prove that $\vect{s} \notin B[\vect{c}; r]$.

\begin{proposition}\label{propNegBallSmallEnough}
If $D(\vect{x}) > 0$, $\vect{x} \in B[\vect{y}; r]$, and $r < \frac{1}{4} D(\vect{x})$, then $\vect{s} \notin B[\vect{y}; r]$.
\end{proposition}

\begin{proof}
By Prop. \ref{propDefectDiffUppEstim},
$\rho(\vect{y}; \vect{x}) \lesq r < \frac{1}{4} D(\vect{x})$ $\Rarr$
$|D(\vect{y}) - D(\vect{x})| < \frac{1}{2}D(\vect{x})$ $\Rarr$ $D(\vect{y}) > \frac{1}{2} D(\vect{x}) > 0$.

Since $r < \frac{1}{4} D(\vect{x}) < \frac{1}{2} D(\vect{y})$,
Cor. \ref{corBallNegTest} implies $\vect{s} \notin B[\vect{y}; r]$.
\end{proof}

{\ }

\textbf{Covershapes and coverands.}
\textit{Covershape} is a certain way to define the subset of $X$
by its \textit{anchor} $\vect{x} \in X$ and its \textit{size} $r > 0$,
such subset to contain at least $\vect{x}$ and to be contained in $B[\vect{x}; r]$.

\textit{Coverand}, denoted by $C[\vect{x}; r]$, is the covershape defined by given $\vect{x}$ and $r$.

The example of a covershape is
``an intersection of $S \subseteq X$ and a closed ball whose center is in $S$'';
one of corresponding coverands is e.g. $B[\theta; 1] \cap S$ ($\vect{x} = \theta$, $r = 1$).

We distinguish them because, for a given space, we can use a single covershape,
based on the properties of that space, while taking many instances of this ``shape'', which are the coverands.
On the other hand, for one and the same space there are usually many covershapes as well.

Put differently, covershape is a type, and coverand is an object of that type.

{\ }

We consider covershapes with the following 2 properties:

{\ }

\assump{asmpCoverandFiniteCoverByHalf}.
For any coverand $C[\vect{c}; r]$ there is a finite cover by $C[\vect{c}_i; \frac{r}{2}]$, where $\vect{c}_i \in C[\vect{c}; r]$.

{\ }

\assump{asmpCoverInCompact}.
For any coverand $C[\vect{c}; r] = F_0$
let $F_1 = \bigcup\limits_{\vect{x} \in F_0} C[\vect{x}; \frac{r}{2}]$, \textellipsis,
$F_k = \bigcup\limits_{\vect{x} \in F_{k-1}} C[\vect{x}; \frac{r}{2^{k}}]$, \textellipsis
Then $F_{\infty} = \bigcup\limits_{k \in \mathbb{Z}_+} F_k$ has SDN property.

{\ }

Note that $F_{k-1} \subseteq F_k$.
Also, $F_{\infty} \subseteq B[\vect{c}; 2r]$, ---
$\forall \vect{x} \in F_{\infty}$: $\vect{x} \in F_k$ for some $k$,
hence $\vect{x} \in C[\vect{c}_{k}; \frac{r}{2^k}] \subseteq B[\vect{c}_{k}; \frac{r}{2^k}]$.
In turn, $\vect{c}_{k} \in C[\vect{c}_{k-1}; \frac{r}{2^{k-1}}] \subseteq B[\vect{c}_{k-1}; \frac{r}{2^{k-1}}]$, \textellipsis,
$\vect{c}_1 \in B[\vect{c}; r]$. Thus

\cenlin{$\rho(\vect{x}; \vect{c}) \lesq \rho(\vect{x}; \vect{c}_k) + \rho(\vect{c}_k; \vect{c}_{k-1}) + ...
+ \rho(\vect{c}_1; \vect{c}) \lesq \frac{r}{2^k} + \frac{r}{2^{k-1}} + ... + r \lesq 2r$}

{\ }

One general way to covershapes and corresponding coverands
that satisfy \assumpref{asmpCoverandFiniteCoverByHalf}--\assumpref{asmpCoverInCompact}
lies in considering any $S \subset X$ which is proper under the same metric $\rho$:
$\forall \vect{x} \in S$, $\forall r > 0$
$B_S[\vect{x}; r] := B[\vect{x}; r] \cap S$ is compact.
Then a closed ball in $S$ is a covershape in $X$, and for any given $\vect{x} \in S$, $r > 0$
$B_S[\vect{x}; r]$ is the coverand.
\assumpref{asmpCoverandFiniteCoverByHalf} holds because $\{ B(\vect{y}; \frac{r}{2}) \cap S \}_{\vect{y} \in B_S[\vect{x}; r]}$
is the open (in $S$) cover of $B_S[\vect{x}; r]$, which, due to properness of $S$, has a finite subcover
$\{ B(\vect{y}_i; \frac{r}{2}) \cap S \}_{i=1}^n$,
so much the more $\{ B_S[\vect{y}_i; \frac{r}{2}] \}_{i=1}^n$ covers $B_S[\vect{x}; r]$;
\assumpref{asmpCoverInCompact} holds too because $F_{\infty} \subseteq B_S[\vect{x}; 2r]$, which is compact.

\section{RC-algorithm}\label{secRCalgo}

\forceparindent
(RC is \textit{Refining Cover}.)
We assume that its ``executor'' calculates $D(\vect{x})$ at given $\vect{x}$,
builds a finite cover of $C[\vect{x}; r]$ from \assumpref{asmpCoverandFiniteCoverByHalf} etc.,
and completes these actions in a finite time.

{\ }

For the sake of simplicity, we add one more assumption, probably the most ``restrictive''
(and thus reducing the generality of our approach) one:

{\ }

\assump{asmpInitialCoverandKnown}. Some coverand $C[\vect{c}; r] \ni \vect{s}$ is known.

{\ }

Let $\delta > 0$ be any precision chosen in advance;
our goal is to obtain $\vect{x} \in X$ such that $\rho(\vect{x}; \vect{s}) < \delta$.

\noindent
\dotfill

\textbf{Step 0.} Let $k := 1$, $\mathfrak{C}_0 := \{ C[\vect{c}; r] \}$, and $r_0 := r$.
Also, choose $D = D_{\infty}$ or $D = D_1$.

{\ }

\textbf{Step 1.} Let $\mathfrak{C}_k := \nullset$.

For each coverand $C = C[\vect{y}; r_{k-1}] \in \mathfrak{C}_{k-1}$, where $r_{k-1} = \frac{r}{2^{k-1}}$,
by \assumpref{asmpCoverandFiniteCoverByHalf} there is the finite cover $\mathfrak{C}$ of $C$, which
consists of the coverands $C' = C[\vect{z}; r_k]$, where $\vect{z} \in C$ and $r_k = \frac{1}{2} r_{k-1} = \frac{r}{2^k}$.

Consider each $C'$ in turn and test it as the corresponding $B'=B[\vect{z}; r_k]$:
if $D(\vect{z}) \lesq 2r_k$, then add $C'$ to $\mathfrak{C}_k$,
that is, let $\mathfrak{C}_k := \mathfrak{C}_k \cup \{ C' \}$.

Since $\vect{s} \in \bigcup\limits_{C \in \mathfrak{C}_{k-1}} C$, at least 1 coverand $C'$ from these $\mathfrak{C}$
contains $\vect{s}$, thus by Prop. \ref{propBallTestReverse} $C'$ passes the test and appears in $\mathfrak{C}_k$.
Therefore, at the end of this step $\mathfrak{C}_k \ne \nullset$ and $\vect{s} \in \bigcup\limits_{C' \in \mathfrak{C}_k} C'$.

{\ }

\textbf{Step 2.} Let $\vect{c}_k$ be the anchor of the arbitrarily chosen coverand from $\mathfrak{C}_k$.

{\ }

\textbf{Step 3.} 
Let $d_k := r_k + \max\limits_{C[\vect{c}'; r_k] \in \mathfrak{C}_k} \rho(\vect{c}_k; \vect{c}')$. If $d_k < \delta$,
then let $\vect{x} := \vect{c}_k$ and halt; else let $k := k + 1$ and goto Step 1.

\noindent
\dotfill

\begin{proposition}
This algorithm halts after a finite number of iterations, at that $\rho(\vect{x}; \vect{s}) < \delta$.
\end{proposition}

\begin{proof}
By \assumpref{asmpCoverInCompact}, $\bigcup\limits_{\mathfrak{C}_k} \bigcup\limits_{C \in \mathfrak{C}_k} C \subseteq F_{\infty}$,
which has SDN property. Hence $\exists \veps > 0$:
$\vect{z} \in F_{\infty}$, $D(\vect{z}) < \veps$ imply $\rho(\vect{z}; \vect{s}) < \frac{1}{4} \delta$.
When $r_k = \frac{r}{2^k} < \frac{1}{2}\veps$ $\LRarr$ $k > \log_2 \frac{2r}{\veps}$,
for the coverands $C[\vect{c}_k; r_k]$ and any $C[\vect{c}'; r_k]$ to be in $\mathfrak{C}_k$ it is necessary that
$D(\vect{c}_k), D(\vect{c}') \lesq 2 r_k < \veps$, thus
$\rho(\vect{c}_k; \vect{c}') \lesq \rho(\vect{c}_k; \vect{s}) + \rho(\vect{s}; \vect{c}') < \frac{1}{2} \delta$.
Then $\max\limits_{C[\vect{c}'; r_k] \in \mathfrak{C}_k} \rho(\vect{c}_k; \vect{c}') < \frac{1}{2} \delta$ too.

As soon as we reach $k$ such that $r_k < \frac{1}{2} \veps$ and $r_k < \frac{1}{2} \delta$
(the latter holds when $k > \log_2 \frac{2r}{\delta}$), we have $d_k < \frac{1}{2} \delta + \frac{1}{2} \delta = \delta$,
and the algorithm halts with $\vect{x} = \vect{c}_k$, where $C[\vect{c}_k; r_k] \in \mathfrak{C}_k$.

Of course, $d_k < \delta$ may become true for $k$ even smaller than
$\max \{ \log_2 \frac{2r}{\veps}; \log_2 \frac{2r}{\delta} \}$.

Suppose $\vect{s} \in C[\vect{c}'; r_k] \in \mathfrak{C}_k$, then
$\rho(\vect{x}; \vect{s}) \lesq \rho(\vect{c}_k; \vect{c}') + \rho(\vect{c}'; \vect{s}) \lesq d_k < \delta$.
\end{proof}

If we replace Step 3 by

\cenlin{\textbf{Step 3'.} Let $k := k+1$, goto Step 1.}

\noindent
then we get the infinite sequence of $\vect{c}_k \xrarr{k \rarr \infty}{} \vect{s}$.
Indeed, $\forall \delta > 0$ the same reasonings provide $\exists \veps > 0$: $\vect{z} \in F_{\infty}$,
$D(\vect{z}) < \veps$ $\Rarr$
$\rho(\vect{z}; \vect{s}) < \delta$. Then $\forall k \greq k_0$, where $r_{k_0} < \frac{1}{2}\veps$
(e.g. $k_0 = \max \bigl\{ 1; \lfloor \log_2 \frac{2r}{\veps} \rfloor + 1 \bigr\}$),
we have $C[\vect{c}_k; r_k] \in \mathfrak{C}_k$, thus
$D(\vect{c}_k) \lesq 2r_k \lesq 2r_{k_0} < \veps$, so $\rho(\vect{c}_k; \vect{s}) < \delta$.

{\ }

\textbf{Remark.}
One can ``weaken'' \assumpref{asmpCoverandFiniteCoverByHalf} to \assumpref{asmpCoverandFiniteCoverByHalf}$\aleph$,
which requires the cover of $C[\vect{c}; r]$ to consist of no more than $\aleph \greq \aleph_0$
coverands $C[\vect{c}_j; \frac{r}{2}]$, $\vect{c}_j \in C[\vect{c}; r]$,
but then one has to ``strengthen'' the algorithm's executor accordingly,
so that it is able to build such cover and test $\aleph$ coverands in a finite time
(also, $d_k := r_k + \sup\limits_{C[\vect{c}'; r_k] \in \mathfrak{C}_k} \rho(\vect{c}_k; \vect{c}')$).
In case $\card \, X = \aleph$ (or even $\card A = \aleph$, $\ovline{A} = X$) it seems easier for the executor
to verify all $\vect{x} \in X$ for being $\vect{s}$ in a more direct way.

{\ }

As the next section illustrates, in certain spaces,
when $\{ \vect{r}_i \}$ are at specific positions and $\{ t_i \}$ take specific values,
there are ``better''/faster methods to approximate $\vect{s}$ or even obtain it exactly.

\section{Dense sensors and normed spaces}\label{secDenseOnSphereNormSpace}

\forceparindent
Here we consider $R$ consisting of ``much more'' sensors, --- in terms of density in $X$ rather than in terms of cardinality.
On the other hand, the components of the algorithm described above, --- the refining cover and the defect, ---
if needed, become much simpler.

We recall that the set $A \subseteq X$ is called \textit{dense} in the set $B \subseteq X$ if $\ovline{A} \supseteq B$.
In particular, when $\ovline{A} = X$, $A$ is \textit{everywhere dense}.

{\ }

Suppose $R \subset \ovline{R}$. We can assume that the set of sensors is $\ovline{R}$ from the start, because
$\forall \vect{r} \in \ovline{R}$ $t_{\vect{r}} = t_0 + \rho(\vect{s}; \vect{r}) =
\lim\limits_{j \rarr \infty} t_{i_j} = t_0 + \lim\limits_{j \rarr \infty} \rho(\vect{s}; \vect{r}_{i_j})$
for $\forall \vect{r}_{i_j} \xrarr{j \rarr \infty}{} \vect{r}$, $\vect{r}_{i_j} \in R$,
due to continuity of metric; that is, the original sensors uniquely define the moments
when the wave reaches new sensors from the closure.
From now on, $\ovline{R} = R$, or, equivalently, $R$ is closed.

{\ }

The easiest case is when we know that $\vect{s} \in R$: $t_{\vect{s}} = t_0$,
while $\forall \vect{r} \in R$, $\vect{r} \ne \vect{s}$:
$t_{\vect{r}} = t_0 + \rho(\vect{s}; \vect{r}) > t_0$, so $t_{\vect{s}} = \inf\limits_{\vect{r} \in R} t_{\vect{r}}$.
In other words, the solution then is the sensor where $t_{\vect{r}}$ attains its infimum.
In general case, $\vect{a} \in R$ such that $t_{\vect{a}} = \inf\limits_{\vect{r} \in R} t_{\vect{r}} =
t_0 + \inf\limits_{\vect{r} \in R} \rho(\vect{s}; \vect{r})$ is the \textit{best approximant} (BA) of $\vect{s}$ in $R$.

{\ }

SRP is simplified when $R$ is ``complex'' enough to ``get'' the ``shape'' of expanding wave
at some moment(s), and from that shape, in turn, derive the position of the source.
In this section we consider spherical sensor-sets in normed spaces.

{\ }

Precisely, hereinafter in this section

$\clubsuit$ $(X; \rho)$ is a normed space $(X; \| \cdot \|)$ with a strictly convex norm, and $\dim X \greq 2$;

$\clubsuit$ $R = S[\theta; 1]$.

(Also, we assume that we can determine, in a finite time, $\vect{r} \in R$ where $t_{\vect{r}}$ attains its $\inf$ or $\sup$.)

{\ }

\textbf{Case $\| \vect{s} \| < 1$.} Consider $\vect{s} \ne \theta$.

$\forall \vect{x} \in X$, $\| \vect{x} - \vect{s} \| < 1 - \| \vect{s} \|$: $\| \vect{x} \| = \| (\vect{x} - \vect{s}) + \vect{s} \| <
1 - \| s \| + \| s \| = 1$, while for $\vect{b} = \vect{s} / \| \vect{s} \|$: $\| \vect{b} \| = 1$
and $\| \vect{b} - \vect{s} \| = \bigl| 1/\| \vect{s} \| - 1 \bigr| \cdot \| \vect{s} \| = 1 - \| \vect{s} \|$.
Thus $\vect{b}$ is BA of $\vect{s}$ in $R$, and $\rho(\vect{s}; R) = 1 - \| \vect{s} \|$.

Suppose $\vect{u}$ is BA of $\vect{s}$ in $R$.
Then $\| \vect{u} - \vect{s} \| = 1 - \| \vect{s} \|$ and $\| \vect{u} \| = 1$. We have
$\| (\vect{u} - \vect{s}) + \vect{s} \| = \| \vect{u} \| = 1 = 1 - \| \vect{s} \| + \| \vect{s} \| = \| \vect{u} - \vect{s} \| +
\| \vect{s} \|$, hence the strict convexity of $\| \cdot \|$ implies $\vect{u} - \vect{s} = \lambda \vect{s}$
$\LRarr$ $\vect{u} = (1 + \lambda) \vect{s}$, $\lambda \greq 0$. Since $1 = \| \vect{u} \| = (1 + \lambda) \| \vect{s} \|$,
$\vect{u} = \vect{s} / \| \vect{s} \| = \vect{b}$, --- BA of $\vect{s}$ in $R$ is unique.

Now, $\forall \vect{u} \in R$: $\| \vect{u} - \vect{s} \| \lesq \| \vect{u} \| + \| \vect{s} \| = 1 + \| \vect{s} \|$,
while for $\vect{w} = - \vect{s} / \| \vect{s} \|$: $\| \vect{w} - \vect{s} \| =
\bigl| 1 / \| \vect{s} \| + 1 \bigr| \cdot \| \vect{s} \| = 1 + \| \vect{s} \|$.
Thus $\vect{w}$ is the \textit{worst approximant} (WA) of $\vect{s}$ in $R$:
$\rho(\vect{s}; \vect{w}) = \sup\limits_{\vect{r} \in R} \rho(\vect{s}; \vect{r})$.

Analogously, if $\vect{u}$ is WA of $\vect{s}$ in $R$, then $\| \vect{u} - \vect{s} \| = 1 + \| \vect{s} \| =
\| \vect{u} \| + \| - \vect{s} \|$, so $\vect{u} = \lambda (- \vect{s})$,
at that $\lambda \greq 0$; $1 = \| \vect{u} \| = \lambda \| \vect{s} \|$ $\Rarr$ $\vect{u} = - \vect{s} / \| \vect{s} \| =
\vect{w}$, --- WA is unique as well.

Let $t_b = \inf\limits_{\vect{r} \in R} t_\vect{r}$ and $t_w = \sup\limits_{\vect{r} \in R} t_{\vect{r}}$.
We see that $t_b = t_0 + 1 - \| \vect{s} \|$ is attained only at $\vect{b}$
and $t_w = t_0 + 1 + \| \vect{s} \|$ is attained only at $\vect{w}$.
Hence $t_w - t_b = 2 \| \vect{s} \|$, implying $\vect{s} = \frac{1}{2} (t_w - t_b) \vect{b}$.

This method of obtaining $\vect{s}$ works for $\vect{s} = \theta$ too, when $t_{\vect{r}} \equiv t_0 + 1$ $\Rarr$ $t_w - t_b = 0$.

{\ }

\textbf{Case $\| \vect{s} \| \greq 1$.}
It is an easy exercise to show that $\vect{b} = \vect{s} / \| \vect{s} \|$ is the unique BA of $\vect{s}$ in $R$
and $\vect{w} = -\vect{b}$ is the unique WA of $\vect{s}$ in $R$ again.
However, $t_b = \inf\limits_{\vect{r} \in R} t_{\vect{r}} = t_0 + \| \vect{s} \| - 1$ (attained at $\vect{b}$)
and $t_w = \sup\limits_{\vect{r} \in R} t_{\vect{r}} = t_0 + \| \vect{s} \| + 1$ (attained at $\vect{w}$),
which isn't enough to determine $\| \vect{s} \|$.

Since $\dim X \greq 2$, $\exists \vect{r} \in R$: $\vect{r} \ne \vect{b}$ and $\vect{r} \ne \vect{w}$,
with corresponding $t_{\vect{r}} = t_0 + \| \vect{r} - \vect{s} \|$.
We claim that $(\vect{b}; t_b)$, $(\vect{w}; t_w)$, and $(\vect{r}; t_{\vect{r}})$ determine $\vect{s} = d \vect{b}$ uniquely on
the ray $L = \{ d\vect{b} \mid d \greq 1 \}$.

Indeed, $d_1 = \| \vect{s} \|$ satisfies all 3 equations. Assume that there is another solution $d_2 \greq 1$, $d_2 \ne d_1$.
$t_b + t_w = 2(t_0 + \| \vect{s} \|)$, so $t_0 = \frac{1}{2}(t_b + t_w) - \| \vect{s} \|$. We rewrite
$t_{\vect{r}} = \frac{1}{2}(t_b + t_w) - d + \| \vect{r} - d \vect{b} \|$, or
$\| \vect{r} - d \vect{b} \| - d = t_{\vect{r}} - \frac{1}{2}(t_b + t_w)$. Then by assumption

\cenlin{$\| \vect{r} - d_1 \vect{b} \| - d_1 = \| \vect{r} - d_2 \vect{b} \| - d_2$ $\Rarr$
$\bigl| \, \| \vect{r} - d_1 \vect{b} \| - \| \vect{r} - d_2 \vect{b} \| \, \bigr| = |d_1 - d_2| =
\| (d_1 - d_2) \vect{b} \|$}

a) If $\| \vect{r} - d_1 \vect{b} \| - \| \vect{r} - d_2 \vect{b} \| = \| (d_1 - d_2) \vect{b} \|$,
then from $\| \vect{r} - d_1 \vect{b} \| = \| \vect{r} - d_2 \vect{b} \| + \| (d_2 - d_1) \vect{b} \|$
and strict convexity of $\| \cdot \|$ it follows that $\vect{r} - d_2 \vect{b} = \lambda (d_2 - d_1) \vect{b}$,
hence $\vect{r} = \gamma \vect{b}$, which is impossible, because then $\| \vect{r} \| = \| \vect{b} \| = 1$
would require $|\gamma| = 1$ and we would obtain $\vect{r} = \pm \vect{b}$ --- a contradiction.

b) If $\| \vect{r} - d_1 \vect{b} \| - \| \vect{r} - d_2 \vect{b} \| = - \| (d_1 - d_2) \vect{b} \|$,
then $\| \vect{r} - d_2 \vect{b} \| = \| \vect{r} - d_1 \vect{b} \| + \| (d_1 - d_2) \vect{b} \|$
likewise implies a contradiction.

Thus the SRP in $L$ defined by $(\vect{b}; t_b)$, $(\vect{w}; t_w)$, $(\vect{r}; t_{\vect{r}})$ satisfies
\assumpref{asmpSensBound} and \assumpref{asmpSolUnique} (note that $\vect{w}, \vect{r} \notin L$).
We approximate $\vect{s}$ to arbitrary precision using the RC-algorithm, and for that we need

1) The defect $D(\vect{x}) = D_{\infty}(\vect{x})$: since $\tau_{\vect{b}}(\vect{x}) = t_b - \| \vect{x} - \vect{b} \| =
t_0 + \| \vect{s} \| - \| \vect{x} \| = t_w - \| \vect{x} - \vect{w} \| = \tau_{\vect{w}}(\vect{x})$,
we have $D(\vect{x}) = \bigl| \tau_{\vect{r}}(\vect{x}) - \tau_{\vect{b}}(\vect{x}) \bigr| =
\bigl| t_{\vect{r}} - t_b - \| \vect{x} - \vect{r} \| + \| \vect{x} \| - 1  \bigr|$;

2) Covershape is a closed segment on the ray $L$, and the coverand $C[\vect{c}; r] = \{ \vect{c} + u \vect{b} \colon |u| \lesq r \}$,
where $\vect{c} \in L$. Moreover, we can assume that an upper estimate of $\| \vect{s} \|$ is known, $\| \vect{s} \| \lesq M+1$,
and consider only $K = L \cap B[\theta; M+1] \ni \vect{s}$, which is compact.
Then $\forall k \in \mathbb{Z}_+$ the coverands are given explicitly
as $C_{k,i} := \bigl\{ (1 + u) \vect{b} \mid u \in [M\frac{i}{2^k}; M\frac{i+1}{2^k}] \bigr\} \subseteq K$, $i = \ovline{0, 2^k-1}$,
at that $C_{k,i} = C_{k+1,2i} \cup C_{k+1, 2i+1}$. In other words,
\assumpref{asmpCoverandFiniteCoverByHalf}, \assumpref{asmpCoverInCompact}, and \assumpref{asmpInitialCoverandKnown} hold.

Next, we choose small enough $\delta > 0$, run the algorithm,
and obtain $\vect{x} \in K$, $\| \vect{x} - \vect{s} \| < \delta$.

{\ }

These cases are distinguished by $t_w - t_b$, which is $2 \| \vect{s} \| < 2$ when $\| \vect{s} \| < 1$,
and $2$ when $\| \vect{s} \| \greq 1$.

{\ }

\textbf{SDN without relative compactness}. We claim that $A = B[\theta; 1]$ has SDN property.
Indeed, let $\vect{s} \in A$; consider $\delta > 0$ and $\vect{x} \in A$ such that $\| \vect{x} - \vect{s} \| \greq \delta$.
The ``straight line'' $L = \{ \vect{s} + d \vect{v} \mid d \in \mathbb{R} \}$,
where $\vect{v} = (\vect{x} - \vect{s}) / \| \vect{x} - \vect{s} \|$,
intersects $R$ at 2 points: $\vect{u}_+$ for $d = d_+ \greq \| \vect{x} - \vect{s} \|$, and $\vect{u}_-$ for $d = d_- \lesq 0$
(it follows from continuity of $f(d) = \| \vect{s} + d \vect{v} \|$, $f(0) = \| \vect{s} \| \lesq 1$,
$f(d) \greq \bigl| |d| - \| \vect{s} \| \bigr| > 1$ when $|d| > \| \vect{s} \| + 1$, and strict convexity of $\| \cdot \|$).

\cenlin{$D_{\infty}(\vect{x}) \greq \bigl| \tau_{\vect{u}_+}(\vect{x}) - \tau_{\vect{u}_-}(\vect{x}) \bigr| =
\Bigl| \bigl[ \| \vect{s} - \vect{u}_+ \| - \| \vect{s} - \vect{u}_- \| \bigr] -
\bigl[ \| \vect{x} - \vect{u}_+ \| - \| \vect{x} - \vect{u}_- \| \bigr] \Bigr| =$}

\cenlin{$ = \Bigl| \bigl[ d_+ - (-d_-) \bigr] -
\bigl[ (d_+ - \| \vect{x} - \vect{s} \|) -  (\| \vect{x} - \vect{s} \| - d_-) \bigr] \Bigr| =
2 \| \vect{x} - \vect{s} \| \greq 2 \delta$}

\noindent
thus $\inf\limits_{\vect{x} \in A \backsl B(\vect{s}; \delta)} D_{\infty} (\vect{x}) \greq 2 \delta$.
If $\dim X = \infty$, $A$ is not relatively compact (\cite[8.30, 8.28]{giles1987}).

\section{Refining $\veps$-neighborhood-covers of compact sets}

\forceparindent
In this section we keep \assumpref{asmpSensBound}, \assumpref{asmpSolUnique},
but discard \assumpref{asmpCoverandFiniteCoverByHalf}, \assumpref{asmpCoverInCompact}, \assumpref{asmpInitialCoverandKnown},
and replace them by

\resetAssumpCount{2}

{\ }

\assump{asmpApproxCompacts}.
There is a known $\{ K_n \}_{n \in \mathbb{N}}$, $K_n \subseteq X$, such that $K_n$ are compact and
$\rho(\vect{s}; K_n) \xrarr{n \rarr \infty}{} 0$.

{\ }

\assump{asmpApproxCompactsInSDN}.
$\bigcup\limits_{n \in \mathbb{N}} K_n \subseteq A \ni \vect{s}$, $\diam A \lesq r$,
where $A$ has SDN property and $r > 0$ is known.

{\ }

For example, let $(X; \rho) = l_p$, $1 \lesq p < \infty$, and $\| \vect{s} \| \lesq M$, at that $M$ is known.
Let the finite-dimensional subspaces
$L_n := \bigl\{ (x_1; ...; x_n; 0; 0; ...) \mid x_j \in \mathbb{R}, j = \ovline{1, n} \bigr\}$, $n \in \mathbb{N}$.
It is well known that
$\vect{s}_n = (s_1; ...; s_n; 0; 0; ...)$ is BA of $\vect{s}$ in $L_n$,
$\rho(\vect{s}; L_n) = \inf\limits_{\vect{u} \in L_n} \rho(\vect{s}; \vect{u}) = \rho(\vect{s}; \vect{s}_n) = \bigl( \sum\limits_{j=n+1}^{\infty} |s_j|^p \bigr)^{\frac{1}{p}}
\xrarr{n \rarr \infty}{} 0$, and $\| \vect{s}_n \| \lesq \| \vect{s} \| \lesq M$.
If the arrangement of the sensors $R$ provides SDN property of $B[\theta; M]$, then
$K_n = B[\theta; M] \cap L_n$ satisfy \assumpref{asmpApproxCompacts} and \assumpref{asmpApproxCompactsInSDN}
with $A = B[\theta; M]$, $r = 2M$.

\begin{definition}
Let $A \subseteq X$ and $\veps > 0$.
$\{ C_j \}_{j \in J}$, where $C_j \subseteq X$, is called an
\emph{$\veps$-neighborhood-cover} of $A$ if $\forall \vect{x} \in X$, $\rho(\vect{x}; A) \lesq \veps$:
$\exists j \in J$: $\vect{x} \in C_j$.
\end{definition}

In other words, $\bigl( \bigcup\limits_{\vect{x} \in A} B[\vect{x}; \veps] \bigr) \cup
\{ \vect{x} \in X \mid \rho(\vect{x}; A) = \veps \} \subseteq \bigcup\limits_{j \in J} C_j$.

\begin{proposition}\label{propNeighCoverCompact}
For any compact $K \subseteq X$ and any $\veps > 0$
there exists a finite $\veps$-neighborhood-cover $\bigl\{ B[\vect{c}_j; 2\veps] \bigr\}_{j=1}^m$
of $K$, at that $\vect{c}_j \in K$.
\end{proposition}

\begin{proof}
$K$ is \textit{totally bounded} (\cite[8.28]{giles1987}), so there is a finite \textit{$\veps$-net} (\cite[8.24]{giles1987})
$\{ \vect{c}_j \}_{j=1}^m \subseteq K$ for $K$.
Then $\bigl\{ B[\vect{c}_j; 2\veps] \bigr\}_{j=1}^m$ is a sought cover:
indeed, $\forall \vect{x} \in X$ such that $\rho(\vect{x}; K) \lesq \veps$, by Weierstrass theorem,
the continuous $f(\vect{u}) = \rho(\vect{x}; \vect{u})$
attains its infimum on $K$ at some $\vect{u}$, $\rho(\vect{x}; \vect{u}) = \rho(\vect{x}; K)$,
and by definition of $\veps$-net $\exists \vect{c}_j$:
$\rho(\vect{u}; \vect{c}_j) \lesq \veps$. Hence
$\rho(\vect{x}; \vect{c}_j) \lesq \rho(\vect{x}; \vect{u}) + \rho(\vect{u}; \vect{c}_j) \lesq 2\veps$.
\end{proof}

Our goal here is to obtain (or rather ``construct'') $\{ \vect{x}_n \}_{n \in \mathbb{N}} \subseteq X$:
$\vect{x}_n \xrarr{n \rarr \infty}{} \vect{s}$.

{\ }

Take $\forall n \in \mathbb{N}$ and consider the sequence of $\frac{r}{2^k}$-neighborhood-covers of $K_n$
from Prop. \ref{propNeighCoverCompact}: $\forall k \in \mathbb{N}$

\cenlin{$\mathfrak{C}_{n, k} = \bigl\{ B[\vect{c}^{(n)}_{k,j}; \frac{r}{2^{k-1}}] \bigr\}_{j=1}^{m_{n, k}}$,
$\vect{c}^{(n)}_{k, j} \in K_n$}

Let $r_k = \frac{r}{2^{k-1}}$. The balls from $\mathfrak{C}_{n, k}$ that pass the test are the elements of

\cenlin{$\mathfrak{S}_{n, k} = \bigl\{ B[\vect{c}; r_k] \in \mathfrak{C}_{n, k} \mid D(\vect{c}) \lesq 2r_k \bigr\}$}

$\diam A \lesq r$ $\Rarr$ $\forall \vect{x} \in A$ $\vect{s} \in B[\vect{x}; r]$,
and Prop. \ref{propBallTestReverse} implies $\mathfrak{S}_{n, 1} = \mathfrak{C}_{n, 1}$.

As for $k > 1$, there are 2 alternatives:

1) if $\vect{s} \in K_n$, then $\forall k \in \mathbb{N}$ $\exists B[\vect{c}; r_k] \in \mathfrak{C}_{n, k}$:
$\vect{s} \in B[\vect{c}; r_k]$
$\stackrel{\mathrm{Prop. \ref{propBallTestReverse}}}{\Rarr}$ $B[\vect{c}; r_k] \in \mathfrak{S}_{n, k} \ne \nullset$.

2) if $\vect{s} \notin K_n$, then $\rho(\vect{s}; K_n) > 0$ due to compactness, thus by SDN
$\veps = \inf\limits_{\vect{x} \in K_n} D(\vect{x}) > 0$.
As soon as $k$ is big enough so that $r_k < \veps/2$, $\forall \vect{x} \in K_n$ $D(\vect{x}) \greq \veps > 2r_k$,
and by Cor. \ref{corBallNegTest} $\mathfrak{S}_{n, k} = \nullset$.

Let $\mu_n = \min \{ k \in \mathbb{N} \mid \mathfrak{S}_{n, k+1} = \nullset \}$
($\mu_n = \infty$ if $\vect{s} \in K_n$), then $\forall k < \mu_n + 1$: $\mathfrak{S}_{n, k} \ne \nullset$.
$\mu_n \greq 1$.

Let $\nu = \min \{ n \in \mathbb{N} \mid \mu_n = \infty \}$
($\nu = \infty$ if $\vect{s} \notin \bigcup\limits_{n \in \mathbb{N}} K_n$), then $\forall n < \nu$: $\mu_n < \infty$.

{\ }

Now we proceed to the algorithm that defines $\{ \vect{x}_n \}_{n \in \mathbb{N}}$:

\noindent
\dotfill

Let $n = 1$. Build $\mathfrak{C}_{1, k}$, $\mathfrak{S}_{1, k}$,
let $\vect{x}_k$ be the center of the arbitrarily chosen ball from $\mathfrak{S}_{1, k}$,
consecutively for $k = 1, 2, \ldots$,
until $\mathfrak{S}_{1, k} = \nullset$ ($k = \mu_1 + 1$ then) or, if that never happens, eternally (when $\mu_1 = \infty$).

\textellipsis

If $\mu_n < \infty$, let $\vect{x}_n := \vect{x}_{n - 1 + \mu_n}$
(that is, replace it by the center of some ball from $\mathfrak{S}_{n, \mu_n}$, part
of the ``finest'' cover of $K_n$ preceding the cover that fails the test entirely) and
discard $\vect{x}_{n+1}$, $\vect{x}_{n+2}$, \textellipsis, $\vect{x}_{n - 1 + \mu_n}$.
Increment $n := n+1$.
Build $\mathfrak{C}_{n, k}$, $\mathfrak{S}_{n, k}$,
let $\vect{x}_{n - 1 + k}$ be the center of the arbitrarily chosen ball from $\mathfrak{S}_{n, k}$,
for $k = 1, 2, \ldots$, until $\mathfrak{S}_{n, k} = \nullset$ ($k = \mu_n + 1$) or eternally ($\mu_n = \infty$).

\textellipsis

\noindent
\dotfill

This algorithm never halts.
One the one hand, while it runs, however long the span $\vect{x}_n$, \textellipsis, $\vect{x}_{n-1+k}$
is at some $(n; k)$-step, at the next step this span can be almost completely erased, leaving $\vect{x}_{n-1+\mu_n}$
in place of $\vect{x}_n$. On the other hand, for a given $j \in \mathbb{N}$ $\vect{x}_j$
has only a finite number of changes (becomes undefined or gets new value).

Eventually this algorithm comes into one of the following mutually exclusive eternal loops:

{\ }

\textbf{Case $\nu = \infty$}. By \assumpref{asmpApproxCompactsInSDN} $\forall \delta > 0$ $\exists \veps > 0$:
if $\vect{x} \in A$ and $D(\vect{x}) < \veps$, then $\rho(\vect{x}; \vect{s}) < \delta$.

$\exists m \in \mathbb{N}$: $\frac{1}{2^m} \veps \lesq r$
(e.g. $m = \max \bigl\{ 1; \lceil \log_2 \frac{\veps}{r} \rceil \bigr\}$).

By \assumpref{asmpApproxCompacts} $\exists N \in \mathbb{N}$: $\forall n \greq N$ $\rho(\vect{s}; K_n) < \frac{1}{2^{m+2}} \veps$.
Consider any such $n$.

It is easy to see that $\exists k \in \mathbb{N}$:
$r_{k} \in [\frac{1}{2^{m+1}} \veps; \frac{1}{2^m} \veps )$.
Indeed, $\log_2 r_1 = \log_2 r \greq \log_2 \veps - m$, therefore for some $k$:
$\log_2 r_k = \log_2 r - (k - 1) \in [\log_2 \veps - m - 1; \log_2 \veps - m ) =
\bigl[ \log_2 (\frac{1}{2^{m+1}} \veps); \log_2 (\frac{1}{2^m} \veps) \bigr)$.

Construction of $\mathfrak{C}_{n,k}$ implies that
$\exists B[\vect{c}; r_k] \in \mathfrak{C}_{n, k}$: $\vect{s} \in B[\vect{c}; r_k]$ ($r_k > 2 \rho(\vect{s}; K_n)$),
and it follows from Prop. \ref{propBallTestReverse} that $B[\vect{c}; r_k] \in \mathfrak{S}_{n, k} \ne \nullset$.
Moreover, $\forall k' < k$: $r_{k'} > r_k$ $\Rarr$ $\mathfrak{S}_{n, k'} \ne \nullset$.
Thus $\mu_n \greq k$.

Since $\vect{x}_n$ is the center of some ball from $\mathfrak{S}_{n, \mu_n}$, we have
$D(\vect{x}_n) \lesq 2 r_{\mu_n} \lesq 2 r_k < \frac{1}{2^{m-1}} \veps \lesq \veps$,
so $\rho(\vect{x}_n; \vect{s}) < \delta$ for $n \greq N$.

{\ }

\textbf{Case $\mu_{\nu} = \infty$}. For the sake of simplicity we
denote $\vect{x}_n := \vect{x}_{\nu - 1 + n} \in K_{\nu}$, $n \in \mathbb{N}$, --- that is,
we skip $\vect{x}_1, \ldots, \vect{x}_{\nu - 1}$.
The reasonings as in Section \ref{secRCalgo} then follow:
due to \assumpref{asmpApproxCompactsInSDN}, $\forall \delta > 0$ $\exists \veps > 0$:
if $\vect{x} \in K_{\nu}$ and $D(\vect{x}) < \veps$, then $\rho(\vect{x}; \vect{s}) < \delta$.
Let $N = \max \bigl\{ 1; \lfloor \log_2 \frac{2r}{\veps} \rfloor + 2 \bigr\}$.
$\forall n \greq N$: $r_n < \frac{1}{2} \veps$,
and $B[\vect{x}_n; r_n] \in \mathfrak{S}_{\nu, n}$ $\Rarr$ $D(\vect{x}_n) \lesq 2 r_n < \veps$ $\Rarr$
$\rho(\vect{x}_n; \vect{s}) < \delta$.

{\ }

Either way, $\vect{x}_n \xrarr{n \rarr \infty}{} \vect{s}$.

\newpage

\section*{Appendix}

\forceparindent
The ancillary files to this paper are the implementation of the RC-algorithm from Section \ref{secRCalgo}
for $\mathbb{R}^m_p$, $1 \lesq p < \infty$ (it is proper,
but can be viewed as a sub/superspace of some non-proper space), in Julia language [\texttt{https://julialang.org}].

{\ }

\cenlin{\small \texttt{\begin{tabular}{|p{4.5cm}|p{1cm}|p{3.5cm}|p{3.5cm}|}
\hline
FILE & SIZE & SHA2-256 & SHA3-256\\
\hline
\hline
space.jl
& 2397
& d991 cca8 6b22 dba8 d055 0fde 89ae db42 a158 dd3d 10ed ab8c cc72 05a0 1aeb b4bf
& 26d7 0de7 72b5 6e8d dffc fa72 465e 87af d904 574f 04c9 1669 0209 a5b5 8bc6 4bba\\
\hline
sr\_ms\_rc.jl
& 1699
& 6197 9bde 1f0a cac8 2bd6 deaa b90d 24cc 7a9a beb3 b798 b533 558b d027 50e8 3054
& 0946 f95d 5288 23b5 4910 8f58 5436 618e 3991 4b14 b392 9d2b 6c01 2f67 ea00 570b\\
\hline
\end{tabular}}}

{\ }

\texttt{sr\_ms\_rc.jl}, with the algorithm itself and the functions it relies on,
almost does not use anything specific to the space.
To change the space, it is enough to modify only \texttt{space.jl}.

{\ }

Typical execution result for $m=2$, $p = 5.6789$, 64 sensors, and $\delta = 0.1$ is shown below.

Note that the number of coverands does not decrease, and the distance error is smaller than the precision $\delta$.

\begin{verbatim}
$ julia sr_ms_rc.jl

Iteration 1: 1 coverands
Iteration 2: 2 coverands
Iteration 3: 4 coverands
Iteration 4: 4 coverands
Iteration 5: 12 coverands
Iteration 6: 12 coverands
Iteration 7: 48 coverands
Iteration 8: 192 coverands
Iteration 9: 576 coverands
Approximated source: Point([7.734374999999999, -9.348958333333337])
Real source: Point([7.701565893029412, -9.36462698238313])
Distance error: 0.032895472278787335
Time: 1.59940107 sec
\end{verbatim}

However, even for $m  = 3$ both memory and time requirements increase significantly.

\end{document}